\documentclass[a4paper,twoside,11pt,english,intlimits]{article}

\usepackage{amsmath,amsthm,amsfonts,amssymb}
\usepackage{babel}
\usepackage[T1]{fontenc}
\usepackage{times,euler,euscript,stmaryrd}
\usepackage{fancyhdr,titlesec,color,url,float,yfonts}
\usepackage[all]{xy}
\usepackage{hyperref}
\usepackage{catex}

\DeclareMathAlphabet{\mathscr}{T1}{pzc}{m}{it}

\titleformat{\section}[block]{\scshape\filcenter\Large}{\thesection.}{.5em}{}
\titleformat{\subsection}[block]{\bfseries\filcenter\large}{\thesubsection.}{.5em}{\medskip}
\titleformat{\subsubsection}[runin]{\bfseries}{\thesubsubsection.}{.5em}{}[.]
\titlespacing{\subsubsection}{0pt}{10pt}{*1}

\newtheoremstyle{ntheorem}%
	{\topsep}{\topsep}{\itshape}{0pt}{\bfseries}{.}{.5em}%
	{\thmnumber{#2.\hspace{.5em}}\thmname{#1}\thmnote{ (#3)}}
	
\newtheoremstyle{ndefinition}%
	{\topsep}{\topsep}{\normalfont}{0pt}{\bfseries}{.}{.5em}%
	{\thmnumber{#2.\hspace{.5em}}\thmname{#1}\thmnote{ (#3)}}

\newtheoremstyle{nremark}%
	{\topsep}{\topsep}{\normalfont}{0pt}{\itshape}{.}{.5em}%
	{\thmnumber{}\thmname{#1}\thmnote{ (#3)}}

\theoremstyle{ntheorem}
  	\newtheorem{theorem}[subsubsection]{Theorem}
  	\newtheorem{proposition}[subsubsection]{Proposition}
	\newtheorem{lemma}[subsubsection]{Lemma}

\theoremstyle{ndefinition}
	\newtheorem{definition}[subsubsection]{Definition}
	\newtheorem{notation}[subsubsection]{Notation}
	\newtheorem{example}[subsubsection]{Example}
	
\theoremstyle{nremark}	
	\newtheorem{remark}[subsubsection]{Remark}

\fancypagestyle{plain}{
\fancyhf{}\fancyfoot[LC,RC]{}
\fancyfoot[LE,RO]{$\thepage$}

}

\UseTips
\SelectTips{eu}{11}

\newdir{ >}{{}*!/-10pt/@{>}}
\newdir{ -}{{}*!/-10pt/@{}}
\newdir{> }{{}*!/+10pt/@{>}}

\makeatletter

\xyletcsnamecsname@{dir4{}}{dir{}}
\xydefcsname@{dir4{-}}{\line@ \quadruple@\xydashh@}
\xydefcsname@{dir4{.}}{\point@ \quadruple@\xydashh@}
\xydefcsname@{dir4{~}}{\squiggle@ \quadruple@\xybsqlh@}
\xydefcsname@{dir4{>}}{\Tttip@}
\xydefcsname@{dir4{<}}{\reverseDirection@\Tttip@}

\xydef@\quadruple@#1{%
	\edef\Drop@@{%
		\dimen@=#1\relax
		\dimen@=.5\dimen@
		\A@=-\sinDirection\dimen@
		\B@=\cosDirection\dimen@
		\setboxz@h{%
			\setbox2=\hbox{\kern3\A@\raise3\B@\copy\z@}%
			\dp2=\z@ \ht2=\z@ \wd2=\z@ \box2
			\setbox2=\hbox{\kern\A@\raise\B@\copy\z@}%
			\dp2=\z@ \ht2=\z@ \wd2=\z@ \box2
			\setbox2=\hbox{\kern-\A@\raise-\B@\copy\z@}%
			\dp2=\z@ \ht2=\z@ \wd2=\z@ \box2
			\setbox2=\hbox{\kern-3\A@\raise-3\B@ \noexpand\boxz@}%
			\dp2=\z@ \ht2=\z@ \wd2=\z@ \box2
		}%
		\ht\z@=\z@ \dp\z@=\z@ \wd\z@=\z@ \noexpand\styledboxz@
	}%
}

\xydef@\Tttip@{\kern2pt \vrule height2pt depth2pt width\z@
	\Tttip@@ \kern2pt \egroup
	\U@c=0pt \D@c=0pt \L@c=0pt \R@c=0pt \Edge@c={\circleEdge}%
	\def\Leftness@{.5}\def\Upness@{.5}%
	\def\Drop@@{\styledboxz@}\def\Connect@@{\straight@{\dottedSpread@\jot}}}
	
\xydef@\Tttip@@{%
	\dimen@=.25\dimen@
 	\B@=\cosDirection\dimen@
	\setboxz@h\bgroup\reverseDirection@\line@ \wdz@=\z@ \ht\z@=\z@ \dp\z@=\z@
	{\vDirection@(1,-1)\xydashl@ \xyatipfont\char\DirectionChar}%
	{\vDirection@(1,+1)\xydashl@ \xybtipfont\char\DirectionChar}%
}

\xydef@\ar@form{
	\ifx \space@\next \expandafter\DN@\space{\xyFN@\ar@form}%
	\else\ifx ^\next \DN@ ^{\xyFN@\ar@style}\edef\arvariant@@{\string^}%
	\else\ifx _\next \DN@ _{\xyFN@\ar@style}\edef\arvariant@@{\string_}%
	\else\ifx 0\next \DN@ 0{\xyFN@\ar@style}\def\arvariant@@{0}%
	\else\ifx 1\next \DN@ 1{\xyFN@\ar@style}\def\arvariant@@{1}%
	\else\ifx 2\next \DN@ 2{\xyFN@\ar@style}\def\arvariant@@{2}%
	\else\ifx 3\next \DN@ 3{\xyFN@\ar@style}\def\arvariant@@{3}%
	\else\ifx 4\next \DN@ 4{\xyFN@\ar@style}\def\arvariant@@{4}%
	\else\ifx \bgroup\next \let\next@=\ar@style
	\else\ifx [\next \DN@[##1]{\ar@modifiers{[##1]}}
	\else\ifx *\next \DN@ *{\ar@modifiers}%
	\else\addLT@\ifx\next \let\next@=\ar@slide
	\else\ifx /\next \let\next@=\ar@curveslash
	\else\ifx (\next \let\next@=\ar@curveinout 
	\else\addRQ@\ifx\next \addRQ@\DN@{\ar@curve@}%
	\else\addLQ@\ifx\next \addLQ@\DN@{\xyFN@\ar@curve}%
	\else\addDASH@\ifx\next \addDASH@\DN@{\defarstem@-\xyFN@\ar@}%
	\else\addEQ@\ifx\next \addEQ@\DN@{\def\arvariant@@{2}\defarstem@-\xyFN@\ar@}%
	\else\addDOT@\ifx\next \addDOT@\DN@{\defarstem@.\xyFN@\ar@}%
	\else\ifx :\next \DN@:{\def\arvariant@@{2}\defarstem@.\xyFN@\ar@}%
	\else\ifx ~\next \DN@~{\defarstem@~\xyFN@\ar@}%
	\else\ifx !\next \DN@!{\dasharstem@\xyFN@\ar@}%
	\else\ifx ?\next \DN@?{\ar@upsidedown\xyFN@\ar@}%
	\else \let\next@=\ar@error
	\fi\fi\fi\fi\fi\fi\fi\fi\fi\fi\fi\fi\fi\fi\fi\fi\fi\fi\fi\fi\fi\fi\fi \next@}

\makeatother

\newcommand{\fl}{\to}

\newcommand{\dfl}{\Rightarrow}
\newcommand{\tfl}{\Rrightarrow}

\newcommand{\ens}[1]{\left\{#1\right\}}

\newcommand{\cl}[1]{\overline{#1}}

\newcommand{\iar}[1]{\left\lfloor#1\right\rfloor}

\renewcommand{\phi}{\varphi}
\renewcommand{\epsilon}{\varepsilon}

\newcommand{\dr}{\partial}

\newcommand{\Br}{\EuScript{B}}
\newcommand{\Cr}{\EuScript{C}}

\newcommand{\Tr}{\EuScript{T}}

\newcommand{\commute}[2]{\ar@{}[#1]|-*+[o][F-]{\scriptscriptstyle{#2}}}

\newcommand{\Ct}[1]{\mathbf{C}{#1}}

\newcommand{\whisk}[1]{\mathbf{W}{#1}}
\newcommand{\Wk}[1]{\mathbf{W}{#1}}

\newcommand{\onecell}[3]{%
\xymatrix@1{
{\scriptstyle #2}
	\ar@<-0.2ex> [r] ^-{#1}
& {\scriptstyle #3}
}%
}
\newcommand{\threecell}[7]{%
\xymatrix@1{
{\scriptstyle #6}
	\ar@/^2ex/ [r] ^{#4} _{}="source"
	\ar@/_2ex/ [r] _{#5} ^{}="target"
	\ar@2 "source" ; "target" ^-{#2} 
& {\scriptstyle #7}
}
\:\overset{#1}{\tfl}\:
\xymatrix@1{
{\scriptstyle #6}
	\ar@/^2ex/ [r] ^{#4} _{}="source"
	\ar@/_2ex/ [r] _{#5} ^{}="target"
	\ar@2 "source" ; "target" ^-{#3} 
& {\scriptstyle #7}
}%
}

\definecolor{orange}{rgb}{1,0.55,0}

\newcommand{\pdf}[1]{\texorpdfstring{$#1$}{1}}

\newcommand{\As}{\mathbf{As}}
\newcommand{\tck}[1]{#1^{\top}}
\newcommand{\ab}[1]{#1_{\text{ab}}}
\newcommand{\abtck}[1]{\ab{\tck{#1}}}
\renewcommand{\tilde}[1]{\widetilde{#1}}
\newcommand{\rep}[1]{\widehat{#1}}
\DeclareMathOperator{\Aut}{Aut}
\newcommand{\Sph}{\mathbf{S}}

\begin{document}

\thispagestyle{empty}
\begin{center}
\textbf{\LARGE Identities among relations \\ for higher-dimensional rewriting systems}

\vspace{4mm}
\begin{tabular}{c c c}
\textbf{\large Yves Guiraud} &$\qquad\qquad$& \textbf{\large Philippe Malbos} \\
INRIA Nancy && Université Lyon 1 \\
LORIA && Institut Camille Jordan\\
yves.guiraud@loria.fr && malbos@math.univ-lyon1.fr
\end{tabular}
\end{center} 

\vspace{2mm}
\begin{em}
\hrule height 1.5pt

\medskip
\noindent \textbf{Abstract --} We generalize the notion of
identities among relations, well known for presentations of
groups, to presentations of n-categories by polygraphs. To
each polygraph, we associate a track n-category,
generalizing the notion of crossed module for groups, in
order to define the natural system of identities among
relations. We relate the facts that this natural system is
finitely generated and that the polygraph has finite
derivation type. 

\noindent \textbf{Support --} This work has been partially
supported by ANR Inval project (ANR-05-BLAN-0267).  

\medskip
\hrule height 1.5pt
\end{em}

\section*{Introduction}

The notion of \emph{identity among relations} originates in the work of Peiffer and Reidemeister, in combinatorial group theory~\cite{Peiffer49,Reidemeister49}. It is based on the notion of \emph{crossed module}, introduced by Whitehead, in algebraic topology, for the classification of homotopy $2$-types~\cite{Whitehead49a,Whitehead49b}. Crossed modules have also been defined for other algebraic structures than groups, such as commutative algebras~\cite{Porter87}, Lie algebras~\cite{KasselLoday82} or categories~\cite{Porter85}. Then Baues has introduced \emph{track $2$-categories}, which are categories enriched in groupoids, as a model of homotopy $2$-type~\cite{BauesDreckmann89,Baues91}, together with \emph{linear track extensions}, as generalizations of crossed modules~\cite{BauesMinian01}. 

There exist several interpretations of identities among relations for presentations of groups: as homological $2$-syzygies~\cite{BrownHuebschmann82}, as homotopical $2$-syzygies~\cite{Loday00} or as Igusa's pictures~\cite{Loday00,KapranovSaito99}. One can also interpret identities among relations as the critical pairs of a group presentation by a convergent word rewriting system~\cite{CremannsOtto96}. This point of view yields an algorithm based on Knuth-Bendix's completion procedure that computes a family of generators of the module of identities among relations~\cite{HeyworthWendsley03}. 

In this work, we define the notion of identities among relations for $n$-categories presented by higher-dimensional rewriting systems called \emph{polygraphs}~\cite{Burroni93}, using notions introduced in~\cite{GuiraudMalbos09}. Given an $n$-polygraph $\Sigma$, we consider the free \emph{track $n$-category} $\Sigma^{\top}$ generated by $\Sigma$, that is, the free $(n-1)$-category enriched in groupoid on $\Sigma$. We define \emph{identities among relations for $\Sigma$} as the elements of an \emph{abelian natural system} $\Pi(\Sigma)$ on the $n$-category $\cl{\Sigma}$ it presents. For that, we extend a result proved by Baues and Jibladze~\cite{BauesJibladze02} for the case $n=2$. 

\bigskip
\noindent{\bf Theorem~\ref{LinearTrackExtension}.}
\emph{A track $n$-category $\Tr$ is abelian if and only if there exists a unique (up to isomorphism) abelian natural system $\Pi(\Tr)$ on $\cl{\Tr}$ such that $\rep{\Pi(\Tr)}$ is isomorphic to $\Aut^{\Tr}$.}

\bigskip
\noindent We define $\Pi(\Sigma)$ as the natural system associated by that result to the abelianized track $n$-category $\abtck{\Sigma}$. In Section~\ref{identitiesAmongRelations}, we give an explicit description of the natural system $\Pi(\Sigma)$. \\

\noindent
Then, in Section~\ref{GeneratingIdentitiesAmongRelations}, we interpret generators of $\Pi(\Sigma)$ as elements of a \emph{homotopy basis} of the track $n$-category $\tck{\Sigma}$, see~\cite{GuiraudMalbos09}. More precisely, we prove:

\bigskip
\noindent{\bf Theorem~\ref{mainTheorem2}.}
\emph{If an $n$-polygraph $\Sigma$ has finite derivation type then the natural system $\Pi(\Sigma)$ is finitely generated.}

\bigskip
\noindent To prove this result, we give a way to compute generators of $\Pi(\Sigma)$ from the critical pairs of a convergent polygraph $\Sigma$. Indeed, there exists, for every critical branching $(f,g)$ of $\Sigma$, a confluence diagram:
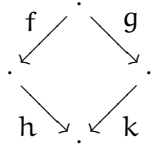
\[
\xymatrix@dr {
\cdot
	\ar [r] ^-{g} 
	\ar [d] _-{f}
& \cdot
	\ar [d] ^-{k}
\\
\cdot
	\ar [r] _-{h}
&
\cdot
}
\]
An $(n+1)$-cell filling such a diagram is called a \emph{generating confluence} of $\Sigma$. It is proved in~\cite{GuiraudMalbos09} that the generating confluences of $\Sigma$ form a homotopy basis of $\tck{\Sigma}$. We show here that they also form a generating set for the natural system $\Pi(\Sigma)$ of identities among relations. 

\section{Preliminaries}
\label{sectionDefinitions}

In this section, we recall several notions from~\cite{GuiraudMalbos09}: presentations of $n$-categories by polygraphs~(\ref{subsectionCat}), rewriting properties of polygraphs~(\ref{subsectionRewriting}), track $n$-categories and homotopy bases~(\ref{subsectionHomotopy}). 

\subsection{Higher-dimensional categories and polygraphs}
\label{subsectionCat}

We fix an $n$-category $\Cr$ throughout this section.

\subsubsection{Notations}

We denote by $\Cr_k$ the set (and the $k$-category) of $k$-cells of~$\Cr$. 
If $f$ is in $\Cr_k$, then $s_i(f)$ and $t_i(f)$ respectively denote the $i$-source and $i$-target of $f$; 
we drop the suffix~$i$ when $i=k-1$. 
The source and target maps satisfy the \emph{globular relations}: 
\begin{equation}
\label{globular relations}
s_i s_{i+1} \:=\: s_i t_{i+1} 
\qquad\text{and}\qquad
t_i s_{i+1} \:=\: t_i t_{i+1}.
\end{equation}
If $f$ and $g$ are $i$-composable $k$-cells, that is when $t_i(f)=s_i(g)$, we denote by $f\star_i g$ their $i$-composite $k$-cell. We also write $fg$ instead of $f\star_0 g$. The compositions satisfy the \emph{exchange relations} given, for every $i\neq j$ and every possible cells $f$, $g$, $h$ and $k$, by: 
\begin{equation}
\label{exchange relation}
(f \star_i g) \star_j (h \star_i k) \:=\: (f \star_j h) \star_i (g \star_j k).
\end{equation}
If $f$ is a $k$-cell, we denote by $1_f$ its identity $(k+1)$-cell and, by abuse, all the higher-dimensional identity cells it generates. When $1_f$ is composed with cells of dimension $k+1$ or higher, we simply denote it by~$f$. A $k$-cell $f$ with $s(f)=t(f)=u$ is called a \emph{closed} $k$-cell with \emph{base point} $u$. 

\subsubsection{Spheres}

Let $\Cr$ be an $n$-category and let $k\in\ens{0,\dots,n}$. A \emph{$k$-sphere of $\Cr$} is a pair $\gamma=(f,g)$ of parallel $k$-cells of $\Cr$, that is, with $s(f)=s(g)$ and $t(f)=t(g)$; we call $f$ the \emph{source} of $\gamma$ and $g$ its \emph{target}. We denote by $\Sph\Cr$ the set of $n$-spheres of $\Cr$. An $n$-category is \emph{aspherical} when all of its $n$-spheres have shape $(f,f)$.

\subsubsection{Cellular extensions}

A \emph{cellular extension of $\Cr$} is a pair $\Gamma=(\Gamma_{n+1},\dr)$ made of a set $\Gamma_{n+1}$ and a map $\dr:\Gamma_{n+1}\fl\Sph\Cr$. 
By considering all the formal compositions of elements of $\Gamma$, seen as $(n+1)$-cells with source and target in $\Cr$, one builds the \emph{free $(n+1)$-category generated by $\Gamma$}, denoted by $\Cr[\Gamma]$. 

The \emph{quotient of $\Cr$ by $\Gamma$}, denoted by $\Cr/\Gamma$, is the $n$-category one gets from $\Cr$ by identification of $n$-cells~$s(\gamma)$ and $t(\gamma)$, for every $n$-sphere $\gamma$ of $\Gamma$. We usually denote by $\cl{f}$ the equivalence class of an $n$-cell~$f$ of~$\Cr$ in $\Cr/\Gamma$. We write $f\equiv_{\Gamma}g$ when $\cl{f}=\cl{g}$ holds. 

\subsubsection{Polygraphs}

We define $n$-polygraphs and free $n$-categories by induction on $n$. A \emph{$1$-polygraph} is a graph, with the usual notion of free category. 

An \emph{$(n+1)$-polygraph} is a pair $\Sigma=(\Sigma_n,\Sigma_{n+1})$ made of an $n$-polygraph $\Sigma_n$ and a cellular extension~$\Sigma_{n+1}$ of the free $n$-category generated by $\Sigma_n$. The \emph{free $(n+1)$-category generated by $\Sigma$} and the \emph{$n$-category presented by $\Sigma$} are respectively denoted by $\Sigma^*$ and $\cl{\Sigma}$ and defined by:
\[
\Sigma^* \:=\: \Sigma_n^*[\Sigma_{n+1}]
\qquad\text{and}\qquad
\cl{\Sigma} \:=\: \Sigma_n^*/\Sigma_{n+1}.
\]
An $n$-polygraph $\Sigma$ is \emph{finite} when each set $\Sigma_k$ is finite, $0\leq k\leq n$. Two $n$-polygraphs whose presented $(n-1)$-categories are isomorphic are \emph{Tietze-equivalent}. A property on $n$-polygraphs that is preserved up to Tietze-equivalence is  \emph{Tietze-invariant}.

An $n$-category $\Cr$ is \emph{presented} by an $(n+1)$-polygraph $\Sigma$ when it is isomorphic to $\cl{\Sigma}$. It is \emph{finitely generated} when it is presented by an $(n+1)$-polygraph $\Sigma$ whose underlying $n$-polygraph $\Sigma_n$ is finite. It is \emph{finitely presented} when it is presented by a finite $(n+1)$-polygraph.

\begin{example}
\deftwocell[black]{mu : 2 -> 1}
\deftwocell[gray]{alpha : 3 -> 1}
\deftwocell[white]{aleph : 4 -> 1}
Let us consider the monoid $\As=\ens{a_0,a_1}$ with unit $a_0$ and product $a_1a_1=a_1$. We see $\As$ as a ($1$-)category with one $0$-cell $a_0$ and one non-degenerate $1$-cell $a_1:a_0\fl a_0$. As such, it is presented by the $2$-polygraph $\Sigma_2$ with one $0$-cell $a_0$, one $1$-cell $a_1:a_0\fl a_0$ and one $2$-cell $a_2:a_1a_1\dfl a_1$. Thus $\As$ is finitely generated and presented. In what follows, we use graphical notations for those cells, where the $1$-cell $a_1$ is pictured as a vertical ``string''~\:\twocell{1}\: and the $2$-cell $a_2$ as~\twocell{mu}. 
\end{example}

\subsubsection{Contexts and whiskers}

A \emph{context of $\Cr$} is a pair $(x,C)$ made of an $(n-1)$-sphere $x$ of $\Cr$ and an $n$-cell~$C$ in $\Cr[x]$ such that $C$ contains exactly one occurrence of $x$. We denote by $C[x]$, or simply by $C$, such a context. If $f$ is an $n$-cell which is parallel to $x$, then $C[f]$ is the $n$-cell of $\Cr$ one gets by replacing $x$ by $f$ in $C$. 

Every context $C$ of $\Cr$ has a decomposition
$$
C \:=\: 
	f_n \star_{n-1} ( f_{n-1} \star_{n-2} ( \cdots 
	\star_1 f_1 x g_1 \star_1 
	\cdots ) \star_{n-2} g_{n-1} ) \star_{n-1} g_n, 
$$
where, for every $k$ in $\ens{1,\dots,n}$, $f_k$ and $g_k$ are $k$-cells of $\Cr$. A \emph{whisker of $\Cr$} is a context that admits such a decomposition with $f_n$ and $g_n$ being identities. Every context $C$ of $\Cr_{n-1}$ yields a whisker of $\Cr$ such that $C[f\star_{n-1}g] =C[f] \star_{n-1} C[g]$ holds.

If $\Gamma$ is a cellular extension of $\Cr$, then every non-degenerate $(n+1)$-cell $f$ of $\Cr[\Gamma]$ has a decomposition 
$$
f \:=\: C_1[\phi_1] \star_n \cdots \star_n C_k[\phi_k],
$$
with $k\geq 1$ and, for every $i$ in $\ens{1,\dots,k}$, $\phi_i$ in $\Gamma$ and $C_i$ a context of $\Cr$.

The \emph{category of contexts of $\Cr$} is denoted by $\Ct{\Cr}$, its objects are the $n$-cells of $\Cr$ and its morphisms from $f$ to $g$ are the contexts $C$ of $\Cr$ such that $C[f]=g$ holds. We denote by $\whisk{\Cr}$ the subcategory of~$\Ct{\Cr}$ with the same objects and with whiskers as morphisms. 

\subsubsection{Natural systems}
\label{NaturalSystems}

A \emph{natural system on $\Cr$} is a functor $D$ from $\Ct{\Cr}$ to the category of groups. We denote by $D_u$ and $D_C$ the images of an $n$-cell~$u$ and of a context $C$ of $\Cr$ by the functor $D$. When no confusion arise, we write $C[a]$ instead of $D_C(a)$. A natural system $D$ on $\Cr$ is \emph{abelian} when $D_u$ is an abelian group for every $n$-cell~$u$.

\subsection{Rewriting properties of polygraphs}
\label{subsectionRewriting}

We fix an $(n+1)$-polygraph $\Sigma$ throughout this section.

\subsubsection{Termination}

One says that an $n$-cell $u$ of $\Sigma_n^*$ \emph{reduces} into an $n$-cell $v$ when $\Sigma^*$ contains a non-identity $(n+1)$-cell with source $u$ and target $v$. One says that $u$ is a \emph{normal form} when it does not reduce into an $n$-cell. A \emph{normal form of $u$} is an $n$-cell $v$ which is a normal form and such that $u$ reduces into $v$. A \emph{reduction sequence} is a countable family $(u_n)_{n\in I}$ of $n$-cells such that each $u_n$ reduces into~$u_{n+1}$; it is \emph{finite} or \emph{infinite} when the indexing set $I$ is.

One says that $\Sigma$ \emph{terminates} when it does not generate any infinite reduction sequence. In that case, every $n$-cell has at least one normal form and one can use \emph{Noetherian induction}: one can prove properties on $n$-cells by induction on the length of reduction sequences.

\subsubsection{Confluence}

A \emph{branching} (resp. \emph{confluence}) is a pair $(f,g)$ of $(n+1)$-cells of $\Sigma^*$ with same source (resp. target), considered up to permutation. A branching $(f,g)$ is \emph{local} when $f$ and $g$ contain exactly one generating $(n+1)$-cell of $\Sigma$. It is \emph{confluent} when there exists a confluence $(f',g')$ with $t(f)=s(f')$ and $t(g)=s(g')$. A local branching $(f,g)$ is \emph{critical} when the common source of $f$ and $g$ is a minimal overlapping of the sources of the $(n+1)$-cells contained in $f$ and $g$. A \emph{confluence diagram} of a branching $(f,g)$ is an $(n+1)$-sphere with shape $(f\star_n f',g\star_n g')$, where $(f',g')$ is a confluence. A confluence diagram of a critical branching is called a \emph{generating confluence of $\Sigma$}. 

One says that $\Sigma$ is \emph{(locally) confluent} when each of its (local) branchings is confluent. A local branching $(f,g)$ is \emph{critical} when the common source of $f$ and $g$ is a minimal overlapping of the sources of the generating $(n+1)$-cells of $f$ and $g$. In a confluent $(n+1)$-polygraph, every $n$-cell has at most one normal form. For terminating $(n+1)$-polygraphs, Newman's lemma ensures that local confluence and confluence are equivalent properties~\cite{Newman42}.  

\subsubsection{Convergence}

One says that $\Sigma$ is \emph{convergent} when it terminates and it is confluent. In that case, every $n$-cell $u$ has a unique normal form, denoted by $\rep{u}$. Moreover, we have $u\equiv_{\Sigma_{n+1}}v$ if and only if $\rep{u}=\rep{v}$. As a consequence, a finite and convergent $(n+1)$-polygraph yields a syntax for the $n$-cells of the category it presents, together with a decision procedure for the corresponding word problem.

\begin{example}
The $2$-polygraph $\Sigma_2=(a_0,a_1,a_2)$ presenting $\As$ is convergent and has exactly one critical pair $(a_2a_1,a_1a_2)$, with corresponding generating confluence~$a_3$:
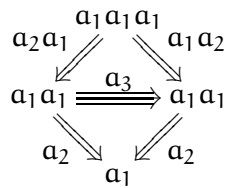
\[
\xymatrix@dr {
{ a_1a_1a_1}
	\ar@2 [r] ^-{a_1a_2} 
	\ar@2 [d] _-{a_2a_1}
& { a_1a_1}
	\ar@2 [d] ^-{a_2}
\\
{ a_1a_1}
	\ar@2 [r] _-{a_2}
	\ar@3 [ur] ^-{a_3}
& { a_1}
}
\]
Alternatively, this $3$-cell $a_3$ can be pictured as follows:
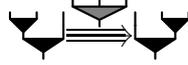
\[
\xymatrix@1{
{ \twocell{(mu *0 1) *1 mu} }
	\ar@3[r] ^-{\twocell{alpha}}
& { \twocell{(1 *0 mu) *1 mu} }
}
\]
In turn, the $3$-polygraph $\Sigma_3=(a_0,a_1,a_2,a_3)$, which is a part of a presentation of the theory of monoids, is convergent and has exactly one critical pair, with corresponding generating confluence~$a_4$: 
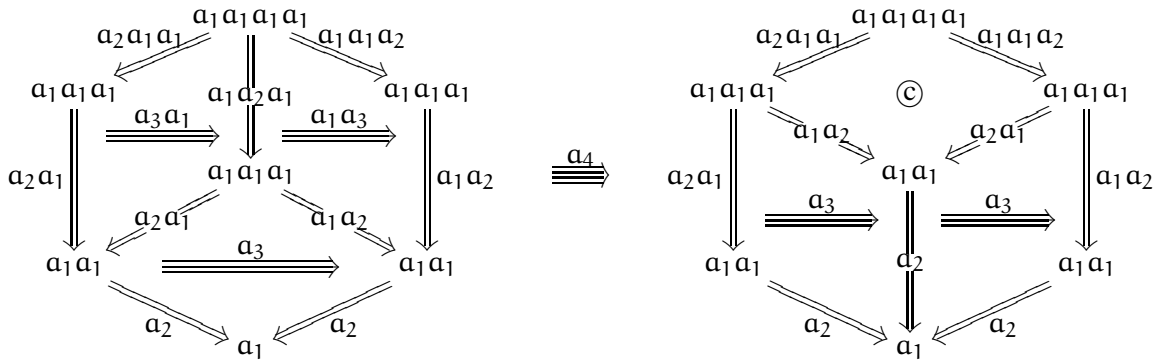
\[
\xymatrix@C=7.5pt@R=15pt{
&& { a_1a_1a_1a_1}
	\ar@2 [dll] _-{a_2a_1a_1}
	\ar@2 [dd] |-{a_1a_2a_1} _(0.75){\hole}="t1" ^(0.75){\hole}="s2"
	\ar@2 [drr] ^-{a_1a_1a_2}
&&&&&&&&&& { a_1a_1a_1a_1}
	\ar@2 [dll] _-{a_2a_1a_1}
	\ar@2 [drr] ^-{a_1a_1a_2}
	\ar@{} [dd] |-{\copyright}	
\\
{ a_1a_1a_1} 
	\ar@2[dd] _-{a_2a_1} ^(0.25){\hole}="s1"
&&&& { a_1a_1a_1}
	\ar@2 [dd] ^-{a_1a_2} _(0.25){\hole}="t2"
&&&&&& { a_1a_1a_1}
	\ar@2[dd] _-{a_2a_1} ^(0.734){\hole}="s3"
	\ar@2[drr] |-{a_1a_2}
&&&& { a_1a_1a_1}
	\ar@2 [dd] ^-{a_1a_2} _(0.734){\hole}="t4"
	\ar@2 [dll] |-{a_2a_1}
\\
&& { a_1a_1a_1}
	\ar@2 [dll] |-{a_2a_1}
	\ar@2 [drr] |-{a_1a_2}
&&&& \strut
	\ar@4 [rr] ^-{a_4}
&&&&&& { a_1a_1}
	\ar@2 [dd] |-{a_2} _(0.26){\hole}="t3" ^(0.26){\hole}="s4"
\\
{ a_1a_1}
	\ar@2 [drr] _-{a_2}
& \strut 
	\ar@3 [rr] ^-{a_3}
&&& { a_1a_1}
	\ar@2 [dll] ^-{a_2}
&&&&&& { a_1a_1} 
	\ar@2 [drr] _-{a_2}
&&&& { a_1a_1}
	\ar@2 [dll] ^-{a_2}
\\
&& { a_1}
&&&&&&&&&& { a_1}
\ar@3 "s1";"t1" ^-{a_3 a_1}
\ar@3 "s2";"t2" ^-{a_1a_3}
\ar@3 "s3";"t3" ^-{a_3}
\ar@3 "s4";"t4" ^-{a_3}
}
\]
In fact, this $4$-cell $a_4$ is Mac Lane's pentagon~\cite{GuiraudMalbos09}: 
\[
\xymatrix@R=5pt{
& {\twocell{ (1 *0 mu *0 1) *1 (mu *0 1) *1 mu }}
	\ar@2 [rr] ^-{\twocell{ (1 *0 mu *0 1) *1 alpha}} 
&& {\twocell{ (1 *0 mu *0 1) *1 (1 *0 mu) *1 mu}}
	\ar@2 [ddr] ^-{\twocell{ (1 *0 alpha) *1 mu}}
\\
&& \ar@4 [dd] ^-{\twocell{aleph}}
\\
{\twocell{(mu *0 2) *1 (mu *0 1) *1 mu}}
	\ar@2 [uur] ^-{\twocell{ (alpha *0 1) *1 mu}}
	\ar@2 [ddrr] _-{\twocell{ (mu *0 2) *1 alpha }}
&&&& {\twocell{(2 *0 mu) *1 (1 *0 mu) *1 mu}}
\\
&& \strut
\\
&& {\twocell{ (mu *0 mu) *1 mu }}
	\ar@2 [uurr] _-{\twocell{ (2 *0 mu) *1 alpha}}
}
\]
\end{example}

\subsection{Track \pdf{n}-categories and homotopy bases}
\label{subsectionHomotopy}

\subsubsection{Track \pdf{n}-categories}

A \emph{track $n$-category} is an $n$-category $\Tr$ whose $n$-cells are invertible, that is, for $n\geq 2$, an $(n-1)$-category enriched in groupoid. In a track $n$-category, we denote by $f^-$ the inverse of the $n$-cell~$f$. A track $n$-category is \emph{acyclic} when, for every $(n-1)$-sphere $(u,v)$, there exists an $n$-cell~$f$ with source $u$ and target $v$. 

The \emph{$n$-category presented} by a track $(n+1)$-category $\Tr$ is the $n$-category $\cl{\Tr}=\Tr_n/\Tr_{n+1}$, where~$\Tr_{n+1}$ is seen as a cellular extension of $\Tr_n$. Two track $(n+1)$-categories are \emph{Tietze-equivalent} if the $n$-categories they present are isomorphic. Given an $n$-category $\Cr$ and a cellular extension $\Gamma$ of $\Cr$, the \emph{track $(n+1)$-category generated by $\Gamma$} is denoted by~$\Cr(\Gamma)$ and defined as follows: 
\[
\Cr(\Gamma) \:=\: \Cr \left[ \Gamma,\; \Gamma^- \right] \big/ \; \text{Inv}(\Gamma)
\]
where $\Gamma^-$ contains the same $(n+1)$-cells as $\Gamma$, with source and target reversed, and $\text{Inv}(\Gamma)$ is made of the $(n+2)$-cells $(\gamma\star_n\gamma^-,1_{s\gamma})$ and $(\gamma^-\star_n\gamma, 1_{t\gamma})$, where $\gamma$ ranges over $\Gamma$. Let us note that, when $f$ and~$g$ are $n$-cells of $\Cr$, we have $f\equiv_{\Gamma}g$ if and only if there exists an $(n+1)$-cell with source $f$ and target~$g$ in~$\Cr(\Gamma)$. When $\Sigma$ is an $(n+1)$-polygraph, one writes $\tck{\Sigma}$ instead of $\Sigma_n^*(\Sigma_{n+1})$.

\subsubsection{Homotopy bases}

Let $\Cr$ be an $n$-category. A \emph{homotopy basis of $\Cr$} is a cellular extension $\Gamma$ of $\Cr$ such that the track $(n+1)$-category $\Cr(\Gamma)$ is acyclic or, equivalently, when the quotient $n$-category $\Cr/\Gamma$ is aspherical or, again equivalently, when every sphere $(f,g)$ of $\Cr$ satisfies $f\equiv_{\Gamma}g$. 

\begin{lemma}[Squier's fundamental confluence lemma]
\label{SquierConfluenceLemma}
Let $\Sigma$ be a convergent $n$-polygraph. The generating confluences of $\Sigma$ form a homotopy basis of $\Sigma^\top$.
\end{lemma}

\begin{remark}
A complete proof of Lemma~\ref{SquierConfluenceLemma} is given in~\cite{GuiraudMalbos09}. Squier has proved the same result for presentations of monoids by word rewriting systems~\cite{Squier87,Squier94}. When formulated in terms of homotopy bases, Squier's result is a subcase of the case $n=2$ of Lemma~\ref{SquierConfluenceLemma}.
\end{remark}

\begin{example}
The $2$-polygraph $\Sigma_2=(a_0,a_1,a_2)$ presenting $\As$ has exactly one generating confluence~$a_3$ and, thus, this $3$-cell forms a homotopy basis of the track $2$-category $\tck{\Sigma_2}$. The $3$-polygraph $\Sigma_3=(a_0,a_1,a_2,a_3)$ also has exactly one generating confluence $a_4$, with Mac Lane's pentagon as shape, which forms a homotopy basis of the track $3$-category $\tck{\Sigma}_3$. 

The resulting $4$-polygraph $\Sigma_4=(a_0,a_1,a_2,a_3,a_4)$ is a part of a presentation of the theory of monoidal categories. In~\cite{GuiraudMalbos09}, Mac Lane's coherence theorem is reformulated in terms of homotopy bases and proved by an application of Lemma~\ref{SquierConfluenceLemma} to a convergent $3$-polygraph containing $\Sigma_3$.
\end{example}

\begin{lemma}
\label{LemmaHomotopiesBoucles}
Let $\Tr$ be a track $n$-category and let $\Br$ be a family of closed $n$-cells of $\Tr$. The following
assertions are equivalent: 
\begin{enumerate}
\item The cellular extension $\tilde{\Br} = \ens{ \tilde{\beta}:\beta\fl 1_{s\beta}, \:\beta\in\Br }$ is a homotopy basis of $\Tr$.
\item Every closed $n$-cell $f$ in $\Tr$ can be written
\begin{equation} \label{HomotopyBaseTDFLemmaEq}
f \:=\: \left( g_1 \star_{n-1} C_1\left[ \beta_1^{\epsilon_1} \right] \star_{n-1} g_1^- \right) 
	\:\star_{n-1}\: \cdots \:\star_{n-1}\: 
	\left( g_k \star_{n-1} C_k \left[ \beta_k^{\epsilon_k} \right] \star_{n-1} g_k^- \right)
\end{equation}
where, for every $i\in\ens{1,\dots,k}$, we have $\beta_i\in\Br$, $\epsilon_i\in\ens{-,+}$, $C_i\in\Wk{\Tr}$ and $g_i\in\Tr_n$.
\end{enumerate}
\end{lemma}

\begin{proof}
Let us assume that $\tilde{\Br}$ is a homotopy basis of $\Tr$ and let us consider a closed $n$-cell $f:w\fl w$ in~$\Tr$. Then, by definition of a homotopy basis, there exists an $(n+1)$-cell $A:f\fl 1_w$ in $\Tr(\tilde{\Br})$. By construction of $\Tr(\tilde{\Br})$, the $(n+1)$-cell $A$ decomposes into
\[
A \:=\: A_1\star_n\cdots\star_n A_k,
\]
where each $A_i$ is an $(n+1)$-cell of $\Tr(\tilde{\Br})$ that contains exactly one generating $(n+1)$-cell of $\Br$. As a consequence, each $A_i$ has shape
\[
g_i\star_{n-1} C_i \left[ \tilde{\beta}_i^{\epsilon_i} \right] \star_{n-1} h_i
\]
with $\beta_i \in \Br$, $\epsilon_i \in \ens{-,+}$, $C_i \in \Wk{\Tr}$ and $g_i, h_i \in \Tr_n$, . By hypothesis on $A$, we have $f=s(A)$, hence:
\[
f \:=\: g_1 \star_{n-1} C_1[s(\beta_1^{\epsilon_1})] \star_{n-1} h_1.
\]
We proceed by case analysis on $\epsilon_1$. If $\epsilon_1=+$, then we have:
\begin{align*}
f \:
	&=\: g_1 \star_{n-1} C_1[\beta_1] \star_{n-1} h_1 \\
	&=\: \left( g_1 \star_{n-1} C_1[\beta_1] \star_{n-1} g_1^- \right) \star_{n-1} \left( g_1 \star_{n-1} h_1 \right) \\
	&=\: \left( g_1 \star_{n-1} C_1[\beta_1] \star_{n-1} g_1^- \right) \star_{n-1} s(A_2).
\end{align*}
And, if $\epsilon_1=-$, we get:
\begin{align*}
f \:
	&=\: g_1 \star_{n-1} h_1 \\
	&=\: \left( g_1 \star_{n-1} C_1[\beta_1^-] \star_{n-1} g_1^- \right) 
		\star_{n-1} \left( g_1 \star_{n-1} C_1[\beta_1] \star_{n-1} h_1 \right) \\
	&=\: \left( g_1 \star_{n-1} C_1[\beta_1^-] \star_{n-1} g_1^- \right) \star_{n-1} s(A_2).
\end{align*}
An induction on the natural number~$k$ proves that $f$ has a decomposition as in~\eqref{HomotopyBaseTDFLemmaEq}.

Conversely, we assume that every closed $n$-cell $f$ in $\Tr$ has a decomposition as in~\eqref{HomotopyBaseTDFLemmaEq}. Then we have $f \equiv_{\tilde{\Br}} 1_{s(f)}$ for every closed $n$-cell $f$ in $\Tr$. Let us consider two parallel $n$-cells $f$ and $g$ in $\Tr$. Then $f\star_{n-1} g^-$ is a closed $n$-cell, yielding $f\star_{n-1} g^- \equiv_{\tilde{\Br}} 1_{s(f)}$. We compose both members by $g$ on the right hand to get $f\equiv_{\tilde{\Br}} g$. Thus~$\tilde{\Br}$ is a homotopy basis of $\Tr$.
\end{proof}

\subsubsection{Finite derivation type}

One says that an $n$-polygraph $\Sigma$ has \emph{finite derivation type} when it is finite and when the track $n$-category $\Sigma^\top$ admits a finite homotopy basis. This property is Tietze-invariant for finite $n$-polygraphs, so that one says that an $n$-category has \emph{finite derivation type} when it admits a presentation by an $(n+1)$-polygraph with finite derivation type.

\begin{lemma}\label{quotientFDT}
Let $\Tr$ be a track $n$-category and let $\Gamma$ be a cellular extension of $\Tr$. If $\Tr$ has finite derivation type, then so does $\Tr/\Gamma$.
\end{lemma}

\begin{proof}
Let $\Br$ be a finite homotopy basis of $\Tr$. Let us denote by $\cl{\Br}$ the cellular extension of $\Tr/\Gamma$ made of one $(n+1)$-cell $\cl{A}$ with source $\cl{f}$ and target $\cl{g}$ for each $(n+1)$-cell $A$ from $f$ to $g$ in $\Br$. Then $\cl{\Br}$ is a homotopy basis of $\Tr/\Gamma$.
\end{proof}

\section{Identities among relations}
\label{finiteDerivationType}

\subsection{Abelian track \pdf{n}-categories}

\begin{definition}
Let $\Tr$ be a track $n$-category. For every $(n-1)$-cell $u$ in $\Tr$, we denote by $\Aut_u^{\Tr}$ the group of closed $n$-cells of $\Tr$ with base $u$. This mapping extends to a natural system  $\Aut^{\Tr}$ on the $(n-1)$-category~$\Tr_{n-1}$, sending a context $C$ of $\Tr_{n-1}$ to the morphism of groups that maps $f$ to $C[f]$. 

A track $n$-category $\Tr$ is \emph{abelian} when, for every $(n-1)$-cell $u$ of $\Tr$, the group $\Aut_u^{\Tr}$ is abelian. The \emph{abelianized} of a track $n$-category $\Tr$ is the track $n$-category denoted by $\ab{\Tr}$ and defined as the quotient of~$\Tr$ by the $n$-spheres $(f\star_{n-1} g, g\star_{n-1} f)$, where $f$ and $g$ are closed $n$-cells with the same base. 
\end{definition}

\begin{lemma}
Each $\Aut_u^{\ab{\Tr}}$ is the abelianized group of $\Aut_u^{\Tr}$. As a consequence, a track $n$-category $\Tr$ is abelian if and only if the natural system $\Aut^{\Tr}$ on $\Tr_{n-1}$ is abelian.
\end{lemma}

\begin{lemma}
Let $\Tr$ be a track $n$-category. For every $n$-cell $g:v\fl u$, the mapping 
$(\cdot)^g$ from $\Aut^{\Tr}_u$ to $\Aut^{\Tr}_v$ and sending $f$ to 
\[
f^g \:=\: g^- \star_{n-1} f \star_{n-1} g
\]
is an isomorphism of groups. Moreover, if $\Tr$ is abelian and $g,h:v\fl u$ are $n$-cells of $\Tr$, then the isomorphisms $(\cdot)^g$ and $(\cdot)^h$ are equal.
\end{lemma}

\begin{proof}
We have:
\[
(1_u)^g 
	\:=\: g^-\star_{n-1} 1_u \star_{n-1} g 
	\:=\: 1_v.
\]
Let $f_1$ and $f_2$ be closed $n$-cells of $\Tr$ with base $u$. Then:
\begin{align*}
(f_1\star_{n-1} f_2)^g
	\:&=\: g^-\star_{n-1} f_1\star_{n-1} f_2 \star_{n-1} g \\
	\:&=\: g^-\star_{n-1} f_1\star_{n-1} g \star_{n-1} g^- \star_{n-1} f_2 \star_{n-1} g \\
	\:&=\: f_1^g\star_{n-1} f_2^g.
\end{align*}
Hence $(\cdot)^g$ is a morphism of groups and it admits $(\cdot)^{g^-}$ as inverse. Now, if $\Tr$ is abelian and $g,h:v\fl u$ are parallel $n$-cells, we have:
\begin{align*}
f^g 
	\:&=\: g^- \star_{n-1} f \star_{n-1} g \\
	\:&=\: (g^- \star_{n-1} h)\star_{n-1} (h^- \star_{n-1} f \star_{n-1} h) \star_{n-1} (h^- \star_{n-1} g) \\
	\:&=\: (h^- \star_{n-1} f \star_{n-1} h) \star_{n-1} (g^- \star_{n-1} h) \star_{n-1} (h^- \star_{n-1} g) \\
	\:&=\: f^h. 
	\tag*{\qedhere}
\end{align*}
\end{proof}

\begin{proposition}\label{TDFabelien_TDF}
If a track $n$-category $\Tr$ has finite derivation type, then so does $\ab{\Tr}$.
\end{proposition}

\begin{proof}
We apply Lemma~\ref{quotientFDT} to the quotient $\ab{\Tr}$ of $\Tr$.
\end{proof}

\subsection{Defining identities among relations}
\label{identitiesAmongRelations} 

\begin{definition}
Let $\Tr$ be a track $n$-category and let $D$ be a natural system on $\cl{\Tr}$. We denote by $\rep{D}$ the natural system on $\Tr_{n-1}$ defined by $\rep{D}_u=D_{\cl{u}}$. A track $n$-category $\Tr$ is \emph{linear} when there exists an abelian natural system $\Pi(\Tr)$ on $\cl{\Tr}$ such that $\rep{\Pi(\Tr)}$ is isomorphic to $\Aut^{\Tr}$. 
\end{definition}

\begin{remark}
If such an abelian natural system $D$ exists, then it is unique up to
isomorphism. Indeed, by definition of $\rep{D}$, we have
$\rep{D}_u=\rep{D}_v$ whenever $u$ and $v$ are $(n-1)$-cells
of $\Tr$ such that $\cl{u}=\cl{v}$ holds. Thus, if $u$ is an
$(n-1)$-cell of $\cl{\Tr}$, then $D_u=\rep{D}_w$ for every
$(n-1)$-cell $w$ of $\Tr$ with $\cl{w}=u$. As a consequence,
if $D$ and~$E$ are abelian natural systems on $\cl{\Tr}$ such that
both $\rep{D}$ and $\rep{E}$ are isomorphic to $\Aut^{\Tr}$,
then $D$ and~$E$ are isomorphic. 
\end{remark}

\begin{theorem}
\label{LinearTrackExtension}
A track $n$-category is abelian if and only if it is linear.
\end{theorem}

\begin{proof}
If $\Tr$ is linear, then each group $\Aut^{\Tr}_u$ is isomorphic to an abelian group. Thus $\Tr$ is abelian.

Conversely, let us assume that $\Tr$ is abelian and let us define the abelian natural system $\Pi(\Tr)$ on $\cl{\Tr}$. For an $(n-1)$-cell $u$ of $\cl{\Tr}$, the abelian group $\Pi(\Tr)_u$ is defined as follows, by generators and relations:
\begin{itemize}
\item It has one generator $\iar{f}$ for every $n$-cell $f:a\fl a$ with $\cl{a}=u$. 
\item Its defining relations are:
\begin{description}
\item[i)] $\iar{f \star_{n-1} g} = \iar{f} + \iar{g}$, for $f:a\fl a$ and $g:a\fl a$ with $\cl{a}=u$; 
\item[ii)] $\iar{f\star_{n-1} g} = \iar{g\star_{n-1} f}$, for $f:a\fl b$ and $g:b\fl a$ with $\cl{a}=\cl{b}=u$.
\end{description}
\end{itemize}
If $u$ and $u'$ are $(n-1)$-cells of $\cl{\Tr}$ and if $C$ is a context of $\cl{\Tr}$ from $u$ to $u'$, then the action
\[ 
\Pi(\Tr)_C \::\: \Pi(\Tr)_u \:\longrightarrow\: \Pi(\Tr)_{u'}
\]
is defined, on a generator $\iar{f}$, with $f$ a closed $n$-cell of $\Tr$ with base $a$ such that $\cl{a}=u$, by 
\[
C\iar{f} \:=\: \iar{ B[f] },
\]
where $B$ is a context of $\Tr_{n-1}$, from $a$ to some $a'$ with $\cl{a}'=u'$, such that $\cl{B}=C$ holds. We note that $B[f]$ is a closed $n$-cell of $\Tr$ with base some $a'$ such that $\cl{a}'=u'$, so that $\iar{B[f]}$ is a generating element of $\Pi(\Tr)_{u'}$. Now, let us check that this action is well-defined, that is, it does not depend on the choice of the representatives $f$ and $B$.

For $f$, we check that $\Pi(\Tr)_C$ is compatible with the relations defining $\Pi(\Tr)_u$. If $f$ and $g$ are closed $n$-cells of $\Tr$ with base $a$ such that $\cl{a}=u$, then we have:
\[
\iar{B[f\star_{n-1} g]} 
	\:=\: \iar{B[f] \star_{n-1} B[g]}
	\:=\: \iar{B[f]} + \iar{B[g]}.
\]
And, for $n$-cells $f:a\fl b$ and $g:b\fl a$, with $\cl{a}=\cl{b}=u$, we have:
\[
\iar{ B[f\star_{n-1} g] } 
	\:=\: \iar{B[f] \star_{n-1} B[g] }
	\:=\: \iar{B[g] \star_{n-1} B[f] }
	\:=\: \iar{ B[g\star_{n-1} f] }.
\]
For $B$, we decompose $C$ in $v\star_{n-2}C'\star_{n-2}w$, where $v$ and $w$ are $(n-1)$-cells of $\cl{\Tr}$ and $C'$ is a whisker of $\cl{\Tr}$. Since $\cl{\Tr}$ and $\Tr_{n-1}$ coincide up to dimension $n-2$, any representative $B$ of $C$ can be written $B=b \star_{n-2}C'\star_{n-2}c$, where $b$ and $c$ are respective representatives of $v$ and $w$ in $\Tr_{n-1}$. As a consequence, it is sufficient (and, in fact, equivalent) to prove that the definition of $\Pi(\Tr)_C$ is invariant with respect to the choice of the representative $B$ of $C$ when $C$ has shape $v\star_{n-2}x$ or $x\star_{n-2}w$.

We examine the case $C=v\star_{n-2}x$, the other one being symmetric. We consider two representatives~$b$ and $b'$ of $v$ in $\Tr_{n-1}$. By definition of $\cl{\Tr}$, there exists an $n$-cell $g:b\fl b'$ in~$\Tr$, as in the following diagram, drawn for the case $n=2$:
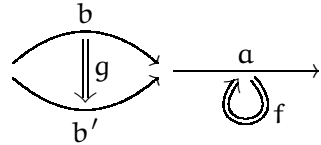
\[
\xymatrix{
\strut 
	\ar@/^3ex/[rr] ^-{b} _-{}="u1"
        \ar@/_3ex/[rr] _-{b'}  ^-{}="u2"
&& \strut
	\ar[rr] ^-{a}="w1" _-{}="w"
&& \strut
\ar@2 "u1";"u2" ^-{g}
\ar @(dr,dl)@{=>} "w" ; "w" ^(0.25){f}
}
\]
Thanks to the exchange relation, we have:
\[
(g\star_{n-2} a) \star_{n-1} (b'\star_{n-2} f) 
	\:=\: g \star_{n-2} f 
	\:=\: (b\star_{n-2} f) \star_{n-1} (g\star_{n-2} a).
\]
Hence: 
\[
b'\star_{n-2} f \:=\: (g^-\star_{n-2} a) \star_{n-1} (b\star_{n-2}f) \star_{n-1} (g\star_{n-2} a).
\]
As, a consequence, one gets, using the second defining relation of $\Pi(\Tr)_{v\star_{n-2} u}$:
\begin{align*}
\iar{b'\star_{n-2} f} 
	\:&=\: \iar{ (g^-\star_{n-2} a) \star_{n-1} (b\star_{n-2} f) \star_{n-1} (g\star_{n-2} a) } \\
	\:&=\: \iar{ (b\star_{n-2} f) \star_{n-1} (g\star_{n-2}a) \star_{n-1} (g^-\star_{n-2}a) } \\
	\:&=\: \iar{ b\star_{n-2} f }.
\end{align*}
Now, let us prove that the abelian natural systems $\rep{\Pi(\Tr)}$ and $\Aut^{\Tr}$ are isomorphic. For an $(n-1)$-cell $u$ of $\Tr$, we define $\Phi_u:\Pi(\Tr)_{\cl{u}}\fl\Aut^{\Tr}_u$ as the morphism of groups given on generators by
\[
\Phi_u(\iar{f}) \:=\: f^g,
\]
where $f$ is a closed $n$-cell of $\Tr$ with base $v$ such that $\cl{v}=\cl{u}$ and $g$ is any $n$-cell of $\Tr$ with source $v$ and target $u$. Let us check that $\Phi_u$ is well-defined. We already know that $\Phi_u$ is independent of the choice of~$g$. Let us prove that this definition is compatible with the relations defining $\Pi(\Tr)_{\cl{u}}$. 

For the first relation, let $f_1$ and $f_2$ be closed $n$-cells of $\Tr$ with base $v$ such that $\cl{v}=\cl{u}$ and let $g:v\fl u$ be an $n$-cell of $\Tr$. Then:
\begin{align*}
\Phi_u(\iar{f_1\star_{n-1}f_2}) 
	\:&=\: (f_1\star_{n-1}f_2)^g \\
	\:&=\: f_1^g \star_{n-1} f_2^g \\
	\:&=\: \Phi_u(\iar{f_1}) \star_{n-1} \Phi_u(\iar{f_2}) \\
	\:&=\: \Phi_u(\iar{f_1}+\iar{f_2}).
\end{align*}
For the second relation, we fix $n$-cells $f_1:v_1\fl v_2$, $f_2:v_2\fl v_1$ and $g:v_1\fl u$, with $\cl{v_1}=\cl{v_2}=\cl{u}$. Then:
\begin{align*}
\Phi_u(\iar{f_1\star_{n-1}f_2})
	\:&=\: (f_1\star_{n-1}f_2)^g \\
	\:&=\: (g^- \star_{n-1} f_1) \star_{n-1} (f_2 \star_{n-1} f_1) \star_{n-1} (f_1^- \star_{n-1} g) \\
	\:&=\: (f_2\star_{n-1} f_1)^{g^-\star_{n-1} f_1} \\
	\:&=\: \Phi_u(\iar{f_2\star_{n-1}f_1}).
\end{align*}
Thus $\Phi_u$ is a morphism of groups from $\Pi(\Tr)_{\cl{u}}$ to $\Aut^{\Tr}_u$. Moreover, it admits $f\mapsto\iar{f}$ as inverse and, as a consequence, is an isomorphism.

Finally, let us prove that $\Phi_u$ is natural in $u$. Let $C$ be a context of $\Tr_{n-1}$ from $u$ to $v$. Let us check that the morphisms of groups $\Phi_v\circ\Pi(\Tr)_{\cl{C}}$ and $\Aut^{\Tr}_C \circ \Phi_u$ coincide. Let $f$ be a closed $n$-cell of $\Tr$ with base point $u'$ such that $\cl{u}'=\cl{u}$. We fix an $n$-cell $g:u'\fl u$ in $\Tr$ and we note that $C[g]$ is an $n$-cell of $\Tr$ with source $C[u']$ and target $C[u]=v$. Then we have:
\begin{align*}
\Phi_v\circ\Pi(\Tr)_{\cl{C}} (\iar{f}) 
	\:&=\: (C[f])^{C[g]} \\
	\:&=\: C[g^-] \star_{n-1} C[f] \star_{n-1} C[g] \\
	\:&=\: C\left[ g^-\star_{n-1} f \star_{n-1} g \right] \\ 
	\:&=\: C[f^g] \\
	\:&=\: \Aut^{\Tr}_C \circ \Phi_u (\iar{f}). \tag*{\qedhere}
\end{align*}
\end{proof}

\begin{remark}
Theorem~\ref{LinearTrackExtension} is proved
in~\cite{BauesDreckmann89,BauesJibladze02} for the case
$n=2$.  
\end{remark}

\begin{definition}
Let $\Sigma$ be an $n$-polygraph. The \emph{natural system of identities among relations of $\Sigma$} is the abelian natural system $\Pi(\abtck{\Sigma})$, which we simply denote by $\Pi(\Sigma)$. If $w$ is an $(n-1)$-cell of $\cl{\Sigma}$, an element of the abelian group $\Pi(\Sigma)_w$ is called an \emph{identity among relations associated to $w$}.
\end{definition}

\subsection{Identities among relations of Tietze-equivalent polygraphs}



\begin{lemma}
\label{LemmaFunctorsTietzeEquivalent}
Let $\Sigma$ and $\Upsilon$ be two Tietze-equivalent $n$-polygraphs. Then there exist $n$-functors 
\[
F \::\: \abtck{\Sigma}\:\fl\:\abtck{\Upsilon}
\qquad\text{and}\qquad
G \::\: \abtck{\Upsilon}\:\fl\:\abtck{\Sigma}
\] 
such that the following two diagrams commute:
\[
\xymatrix{
{\abtck{\Sigma}} 
	\ar[r] ^-{F}
	\ar@{->>} [d] _-{\pi_{\Sigma}}
	\ar@{} [dr] |-{\copyright}
& {\abtck{\Upsilon}}
	\ar@{->>} [d] ^-{\pi_{\Upsilon}}
\\
{\cl{\Sigma}}
	\ar@{=} [r] 
& {\cl{\Upsilon}}
}
\qquad\qquad\qquad
\xymatrix{
{\abtck{\Upsilon}} 
	\ar[r] ^-{G}
	\ar@{->>} [d] _-{\pi_{\Upsilon}}
	\ar@{} [dr] |-{\copyright}
& {\abtck{\Sigma}}
	\ar@{->>} [d] ^-{\pi_{\Sigma}}
\\
{\cl{\Upsilon}}
	\ar@{=} [r] 
& {\cl{\Sigma}}
}
\]
\end{lemma}

\begin{proof}
To simplify notations, we consider that the $(n-1)$-categories $\cl{\Sigma}$ and $\cl{\Upsilon}$ are equal, instead of simply isomorphic. Let us build $F$, the construction of $G$ being symmetric. 

First, we define an $n$-functor $F$ from~$\tck{\Sigma}$ to $\tck{\Upsilon}$. On $i$-cells, with $i\leq n-2$, $F$ is the identity, which makes the diagram commute up to dimension $n-2$ since $\pi_{\Sigma}$ and $\pi_{\Upsilon}$ are also identities on the same dimensions. 

If $a$ is an $(n-1)$-cell in $\Sigma$, we arbitrarily choose an $(n-1)$-cell in $\pi_{\Upsilon}^{-1}\pi_{\Sigma}(a)$ for $F(a)$. Since $F$ is the identity up to dimension $n-2$, we have that the source and target of $F(a)$ are equal to the source and target of $a$, respectively.

Then, $F$ is extended to any $(n-1)$-cell of $\tck{\Sigma}$ by functoriality. Let $\phi:u\fl v$ be an $n$-cell of $\Sigma$. We have, by definition of $F(u)$ and $F(v)$:
\[
\pi_{\Upsilon} \circ F(u) \:=\: \pi_{\Sigma}(u) \:=\: \pi_{\Sigma}(v) \:=\: \pi_{\Upsilon} \circ F(v).
\]
Thus, there exists an $n$-cell from $F(u)$ to $F(v)$ in $\tck{\Sigma}$. We arbitrarily choose $F(\phi)$ to be one of those $n$-cells and, then, we extend $F$ to any $n$-cell of $\tck{\Sigma}$ by functoriality. 

Let $f$ and $g$ be closed $n$-cells in $\tck{\Sigma}$. We have $F(f\star_{n-1} g) = F(f)\star_{n-1} F(g)$ by definition of $F$. As a consequence, $F$ induces a $n$-functor from $\abtck{\Sigma}$ to $\abtck{\Upsilon}$ that satisfies, by construction, the relation $\pi_{\Upsilon}\circ F=\pi_{\Sigma}$. 
\end{proof}

\begin{notation}
We fix two Tietze-equivalent $n$-polygraphs $\Sigma$ and $\Upsilon$, together with $n$-functors $F$ and~$G$ as in Lemma~\ref{LemmaFunctorsTietzeEquivalent}. We denote by $\tilde{G}$ the morphism of natural systems on $\cl{\Sigma}=\cl{\Upsilon}$, from $\Pi(\Upsilon)$ to $\Pi(\Sigma)$, defined by $\tilde{G}(\iar{f})=\iar{G(f)}$.

For every $(n-1)$-cell $w$ in $\abtck{\Sigma}$, we define an $n$-cell $\Lambda_w$ from $w$ to $GF(w)$ in $\abtck{\Sigma}$, by structural induction on $w$. If $w$ is an identity, then $\Lambda_w = 1_w$. Now, let $w$ be an $(n-1)$-cell in $\Sigma$. By hypothesis on~$F$ and $G$, we have:
\[
\pi_{\Sigma} \circ GF(w) \:=\: \pi_{\Upsilon} \circ F(w) \:=\: \pi_{\Sigma} (w).
\]
As a consequence, there exists an $n$-cell from $w$ to $GF(w)$ in $\abtck{\Sigma}$ and we arbitrarily choose $\Lambda_w$ to be such an $n$-cell. Finally, if $w=w_1\star_i w_2$, for some $i\in\ens{0,\dots,n-2}$, then  $\Lambda_w = \Lambda_{w_1}\star_i\Lambda_{w_2}$. If $f:u\fl v$ is an $n$-cell of $\abtck{\Sigma}$, we denote by $\Lambda_f$ the closed $n$-cell with basis $u$ defined by:
\[
\Lambda_f \:=\: f \star_{n-1} \Lambda_v \star_{n-1} GF(f)^- \star_{n-1} \Lambda_u^-.
\]
Finally, we define:
\[
\Lambda_{\Sigma} \:=\: \ens{ \; \iar{\Lambda_{\phi}} \;\big|\; \phi\in\Sigma_n \;}.
\]
\end{notation}

\begin{lemma}
Let $f$ be an $n$-cell in $\abtck{\Sigma}$ with a decomposition
\[
f \:=\: C_1[\phi_1^{\epsilon_1}]\star_{n-1} \dots \star_{n-1} C_k[\phi_k^{\epsilon_k}],
\]
with $\phi_i \in \Sigma_n$, $\epsilon_i \in \ens{-,+}$ and $C_i \in \Wk\Sigma^*$. Then we have:
\begin{equation}
\label{claim1}
\iar{\Lambda_f} \:=\: \sum_{i=1}^k \epsilon_i C_i\iar{\Lambda_{\phi_i}}.
\end{equation}
\end{lemma}

\begin{proof}
Let $f:u\fl v$ and $g:v\fl w$ be $n$-cells in $\abtck{\Sigma}$. We have:
\begin{align*}
\Lambda_{f\star_{n-1} g} \:
	&=\: (f\star_{n-1} g) \star_{n-1} \Lambda_w \star_{n-1} GF(f\star_{n-1} g)^- \star_{n-1} \Lambda^-_u \\
	&=\: f \star_{n-1} \left(g \star_{n-1} \Lambda_w \star_{n-1} GF(g)^- \star_{n-1} \Lambda ^-_v \right)
		\star_{n-1} \Lambda_v \star_{n-1} GF(f)^- \star_{n-1} \Lambda^-_u \\
	&=\: f \star_{n-1} \Lambda_g \star_{n-1} \Lambda_v \star_{n-1} GF(f)^- \star_{n-1} \Lambda^-_u \\
	&=\: f \star_{n-1} \Lambda_g \star_{n-1} f^- \star_{n-1} \Lambda_f. 
\end{align*}
Hence:
\begin{equation}
\label{claim1_1}
\iar{\Lambda_{f\star_{n-1} g}} \:=\: 
\iar{f \star_{n-1} \Lambda_g \star_{n-1} f^- \star_{n-1} \Lambda_f} \:=\: 
\iar{\Lambda_f}+\iar{\Lambda_g}.
\end{equation}
Now, let $f:w\fl w'$ be an $n$-cell and $u$ be an $i$-cell, $i\leq n-1$, of $\abtck{\Sigma}$ such that $u\star_iw$ is defined. Then we have:
\begin{align*}
\Lambda_{u\star_i f} \:
	&=\: (u\star_i f) \star_{n-1} \Lambda_{u\star_iw'} \star_{n-1} GF(u\star_i f)^- \star_{n-1} \Lambda_{u\star_i w}^- \\
	&=\: (u\star_i f) \star_{n-1} (\Lambda_u\star_i \Lambda_{w'}) \star_{n-1} (GF(u)\star_i GF(f)^-) 
		\star_{n-1} (\Lambda_u^-\star_i\Lambda_w^-) \\
	&=\: (u\star_{n-1}\Lambda_u\star_{n-1}GF(u)\star_{n-1}\Lambda_u^-) 
		\star_i (f \star_{n-1} \Lambda_{w'} \star_{n-1} GF(f)^- \star_{n-1} \Lambda_w^-) \\
	&=\: u\star_i \Lambda_f.
\end{align*}
Similarly, we prove that $\Lambda_{f\star_i v}=\Lambda_{f}\star_i v$ if $v$ is an $i$-cell, $i\leq n-1$, such that $w\star_i v$ is defined. As a consequence, we get $\Lambda_{C[f]}=C[\Lambda_f]$, for every whisker $C$ of $\Sigma^*$, hence:
\begin{equation}
\label{claim1_2}
\iar{\Lambda_{C[f]}} \:=\: C\iar{\Lambda_f}.
\end{equation}
We prove~\eqref{claim1} by induction on $k$, using~\eqref{claim1_1} and~\eqref{claim1_2}.
\end{proof}

\begin{lemma}\label{LemmaTransportGenerators}
Let $\Br$ be a generating set for the natural system $\Pi(\Upsilon)$. 
Then the set $\Lambda_{\Sigma} \amalg \tilde{G}(\Br)$ is a
generating set for the natural system $\Pi(\Sigma)$.
\end{lemma}

\begin{proof}
Let $f$ be a closed $n$-cell with basis $w$ in $\tck{\Sigma}$. By definition of $\Lambda_f$, we have:
\[
\iar{f} 
	\:=\: \iar{ \Lambda_f\star_{n-1}\Lambda_w\star_{n-1}GF(f)\star_{n-1}\Lambda_w^-}
	\:=\: \iar{\Lambda_f} + \iar{GF(f)}.
\]
On the one hand, we consider a decomposition of $f$ in generating $n$-cells of $\Sigma_n$: 
\[
f \:=\: C_1[\phi_1^{\epsilon_1}] \star_{n-1} \dots \star_{n-1} C_k[\phi_k^{\epsilon_k}].
\]
Hence:
\[
\iar{\Lambda_f} \:=\: \sum_{i=1}^k \epsilon_i C_i \iar{\Lambda_{\phi_i}}.
\]
On the other hand, the natural system $\Pi(\Upsilon)$ is generated by $\Br$, so that $\iar{F(f)}$ admits a decomposition $\iar{F(f)} = \sum_{j\in J} \eta_j B_j\iar{g_j}$, with $\iar{g_j}\in \Br$. Hence:
\[
\iar{GF(f)} 
	\:=\: \sum_{j\in J} B_j \iar{G(g_j)} 
	\:=\: \sum_{j\in J} B_j [\tilde{G}(\iar{g_j})].
\]
Thus, $\iar{f}$ can be written as a linear combination of elements of $\Lambda_{\Sigma}$ and of $\Br$, proving the result.
\end{proof}

\begin{proposition}\label{lemmaFinitelyGenerated}
Let $\Sigma$ and $\Upsilon$ be two Tietze-equivalent $n$-polygraphs such that $\Sigma_n$ and $\Upsilon_n$ are finite. Then the natural system $\Pi(\Sigma)$ is finitely generated if and only if the natural system $\Pi(\Upsilon)$ is finitely generated. 
\end{proposition}

\begin{proof}
We use Lemma~\ref{LemmaTransportGenerators} with $\Br$ and $\Sigma_n$ finite.
\end{proof}

\subsection{Generating identities among relations}
\label{GeneratingIdentitiesAmongRelations} 

\begin{theorem}
\label{mainTheorem2}
If an $n$-polygraph $\Sigma$ has finite derivation type then the natural system $\Pi(\Sigma)$ is finitely generated. 
\end{theorem}

\begin{proof}
Let us assume that the $n$-polygraph $\Sigma$ has finite derivation type. By Proposition~\ref{TDFabelien_TDF}, the abelian track category $\abtck{\Sigma}$ has finite derivation type. Let $\Br$ be a finite homotopy basis of $\abtck{\Sigma}$ and let $\tilde{B}$ be the set of closed $n$-cells of $\abtck{\Sigma}$ defined by:
\[
\tilde{\Br} \:=\: \ens{ \; s(\beta)\star_{n-1} t(\beta)^- \; \big| \; \beta\in \Br \;}.
\]
By Lemma~\ref{LemmaHomotopiesBoucles}, any closed $n$-cell $f$ in $\abtck{\Sigma}$ can be written 
\[
f \:=\: \big( g_1\star_{n-1} C_1[\beta_1^{\epsilon_1}] \star_{n-1} g_1^- \big) 
	\star_{n-1}\dots\star_{n-1}
	\big( g_k\star_{n-1} C_k[\beta_k^{\epsilon_k}] \star_{n-1} g_k^- \big),
\]
where, for every $i$ in $\ens{1,\dots,k}$, $\beta_i \in \widetilde{\Br}$, $\epsilon_i \in \ens{-,+}$, $C_i \in \Wk\Sigma^*$ and $g_i \in \Sigma^*_n$. As a consequence, for any identity among relations $\iar{f} $ in $\Pi(\Sigma)$, we have: 
\[
\iar{f} 
	\:=\: \sum_{i=1}^k \epsilon_i \iar{g_i \star_{n-1} C_i[\beta_i] \star_{n-1} g_i^-}
	\:=\: \sum_{i=1}^k \epsilon_i C_i\iar{\beta_i}.
\]
Thus, the elements of $\iar{\tilde{B}}$ form a generating set for $\Pi(\Sigma)$.
\end{proof}

\begin{proposition}\label{PropgeneratorPi_2}
For a convergent $n$-polygraph $\Sigma$, the natural system $\Pi(\Sigma)$ is generated by the generating confluences of $\Sigma$.
\end{proposition}

\begin{proof}
By Squier's confluence lemma (Lemma~\ref{SquierConfluenceLemma}), the set of generating confluences of $\Sigma$ forms a homotopy basis of $\Sigma^\top$. Following the proof of Theorem~\ref{mainTheorem2}, we transform it into a generating set for the natural system $\Pi(\Sigma)$.
\end{proof}

\begin{example}\label{ExampleA2}
We consider the $2$-polygraph $\Sigma=(a_0,a_1,a_2)$ presenting the monoid $\As$. Here is a part of the free $2$-category $\Sigma^*$:
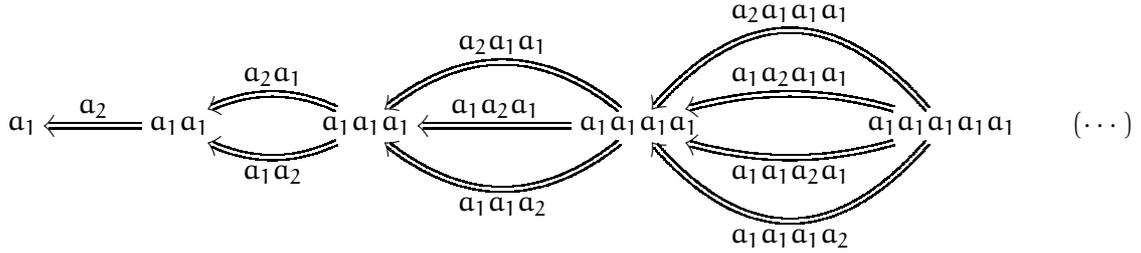
\[
\xymatrix@C=15pt{
{ a_1}
&& { a_1a_1}
	\ar@2[ll] _-{a_2} 
&& { a_1a_1a_1} 
	\ar@2@/_1pc/[ll] _-{a_2a_1}
	\ar@2@/^1pc/[ll] ^-{a_1a_2}
&&& { a_1a_1a_1a_1} 
	\ar@2[lll] _-{a_1a_2a_1}
	\ar@2@/_2pc/[lll] _-{a_2a_1a_1}
	\ar@2@/^2pc/[lll] ^-{a_1a_1a_2}
&&& { a_1a_1a_1a_1a_1}
	\ar@2@/_1pc/[lll] _-{a_1a_2a_1a_1}
	\ar@2@/^1pc/[lll] ^-{a_1a_1a_2a_1} 
	\ar@2@/_3pc/[lll] _-{a_2a_1a_1a_1}
	\ar@2@/^3pc/[lll] ^-{a_1a_1a_1a_2}
& (\cdots)
}
\] 
The $2$-polygraph $\Sigma$ is convergent and has exactly one generating confluence, written with both notations:
\[
\xymatrix@1{
a_2a_1 \star_1 a_2
	\ar@3[r] ^-{a_3}
& a_1a_2 \star_1 a_2
}
\qquad\qquad\qquad
\xymatrix@1{
{\twocell{(mu *0 1) *1 mu}} 
	\ar@3[r] ^-{\twocell{alpha}}
& {\twocell{(1 *0 mu) *1 mu}}
}
\]
Thus the natural system $\Pi(\Sigma)$ on the category $\cl{\Sigma}=\As$ is generated by following the element, where the last equality uses the exchange relation:
\[
\iar{ s(a_3)\star_1 t(a_3)^- } 
	\:=\: \iar{(a_2a_1\star_1 a_2) \star_1 (a_2^- \star_1 a_1a_2^-)}
	\:=\: \iar{a_2a_1\star_1 a_1a_2^-}
	\:=\: \iar{a_2a_2^-}.
\] 
\deftwocell[black]{delta : 1 -> 2}%
The graphical notations, where $\twocell{mu}^-$ is pictured as $\twocell{delta}$, make this last equality more clear:
\[
\iar{ s(\twocell{alpha}) \star_1 t(\twocell{alpha}) ^- }
	\:=\: \iar{ \twocell{(mu *0 1) *1 mu *1 delta *1 (1 *0 delta)} }
	\:=\: \iar{ \twocell{(mu *0 1) *1 (1 *0 delta)} }
	\:=\: \iar{ \twocell{mu}\twocell{delta} }
\]
One can prove the same result by a combinatorial analysis. Indeed, one can note that the minimal $2$-cells from $a_1^{n+1}$ to $a_1^n$ are the $a_1^ia_2a_1^{n-1-i}$, for $i$ in $\ens{0,\dots,n-1}$. Thus, the natural system $\Pi(\Sigma)$ is generated by the following elements, for $n\geq 2$ and $0\leq i<j\leq n-1$:
\[
\iar{g_{i,j}} \:=\: \iar{a_1^ia_2a_1^{n-i-1} \star_1 a_1^ja_2^-a_1^{n-j-1}}.
\]
Then, one uses the exchange relations to get:
\[
g_{i,j} \:=\: 
\begin{cases}
\: a_1^i (a_2a_1\star_1 a_1a_2^-) a_1^{n-i-1} &\text{if }j=i+1 \\
\: a_1^i a_2 a_1^{j-i-2} a_2^- a_1^{n-j-1} &\text{if }j>i+2. 
\end{cases}
\] 
Hence, if $j=i+1$, we have, using the relations defining $\Pi(\Sigma)$ and $\iar{a_1}=0$:
\[
\iar{g_{i,i+1}} \:=\: i\iar{a_1} + \iar{a_2a_1\star_1a_1a_2^-} + (n-i-1)\iar{a_1} \:=\: \iar{a_2a_2^-}.
\]
And, if $j>i+2$, we get:
\[
\iar{g_{i,j}} \:=\: i\iar{a_1} + \iar{a_2} + (j-i-2)\iar{a_1} -\iar{a_2} +(n-j-1)\iar{a_1} \:=\: 0. 
\]
Thus, the natural system $\Pi(\Sigma)$ is generated by one element: $\iar{a_2a_2^-}$.
\end{example}

\begin{small}
\renewcommand{\refname}{References}
\bibliographystyle{amsplain}
\bibliography{../bibliographie}
\end{small}

\end{document}